\documentclass[reqno, 12pt]{amsart}
\usepackage{hyperref}
\usepackage{amssymb,amsmath,amsfonts,amsthm,comment,mathrsfs,times,graphicx}
\usepackage{color}
\usepackage[english]{babel}
\usepackage[all,cmtip]{xy}
\usepackage{lmodern}
\usepackage{enumitem}
\usepackage{geometry}
\newcommand{\eq}[1][r]
{\ar@<-3pt>@{-}[#1]
\ar@<-1pt>@{}[#1]|<{}="gauche"
\ar@<+0pt>@{}[#1]|-{}="milieu"
\ar@<+1pt>@{}[#1]|>{}="droite"
\ar@/^2pt/@{-}"gauche";"milieu"
\ar@/_2pt/@{-}"milieu";"droite"}

\newcommand{\Gal}[1]{\mathrm{Gal}(#1)}

\newcommand{\Z}{\mathbb{Z}}
\newcommand{\Q}{\mathbb{Q}}
\newcommand{\Tor}[2]{\mathrm{Tor}_{#1}(#2)}
\newcommand{\Fr}[2]{\mathrm{Fr}_{#1}(#2)}

\newcommand{\X}[2]{\mathfrak{X}_{(#1)}^{(#2)}}
\newcommand{\XS}[2]{\mathfrak{X}_{S}^{(#1)}(#2)}
\newcommand{\XSp}[1]{\mathfrak{X}_{S}(#1)}
\newcommand{\rank}[1]{\mathrm{rank}_{#1}}
\newcommand{\ind}[2]{\mathrm{Ind}^{#1}_{#2}}
\newcommand{\coind}[2]{\mathrm{Coind}^{#1}_{#2}}
\newcommand{\cker}[1]{\mathcal{K}er_{S}^2(#1, \Z_p(i))}

\newcounter{dummy} \numberwithin{dummy}{section}
\newtheorem{theo}[dummy]{Theorem }
\newtheorem{coro}[dummy]{Corollary}
\newtheorem{lemm}[dummy]{Lemma}
\newtheorem{prop}[dummy]{Proposition}

\newtheorem{rema}[dummy]{Remark}

\newtheorem*{conj}{Conjecture}
\newtheorem*{theo*}{Theorem}

\title{\sc On Greenberg's generalized conjecture}

\begin{document}
	\maketitle
\begin{center}	
		
		{\sc J. Assim$^{(1)}$ and Z. Boughadi$^{(2)}$ }\\
		{\footnotesize $^{(1)}$ Moulay Ismail University of Mekn\`{e}s, Morocco\\
			Team work: Algebraic Theory and Application}\\
		{\footnotesize e-mail: {\it  j.assim@umi.ac.ma}}\\
	{\footnotesize $^{(2)}$ Moulay Ismail University of Mekn\`{e}s, Morocco\\
		Team work: Algebraic Theory and Application }\\
	{\footnotesize e-mail: {\it z.boughadi@edu.umi.ac.ma}}

\end{center}

\thispagestyle{empty}

\hrulefill
\renewcommand{\abstractname}{Abstract}
\begin{abstract}  
	{\footnotesize For a number field $F$ and an odd prime number $p,$ let $\tilde{F}$ be the compositum of all $\Z_p$-extensions of $F$ and $\tilde{\Lambda}$ the associated Iwasawa algebra. Let $G_{S}(\tilde{F})$ be the Galois group over $\tilde{F}$ of the maximal extension which is unramified outside $p$-adic and infinite places. In this paper we study the $\tilde{\Lambda}$-module $\XS{-i}{\tilde{F}}:=H_1(G_S(\tilde{F}), \Z_p(-i))$ and its relationship with $X(\tilde{F}(\mu_p))(i-1)^\Delta,$ the $\Delta:=\Gal{\tilde{F}(\mu_{p})/\tilde{F}}$-invariant of the Galois group over $\tilde{F}(\mu_{p})$ of the maximal abelian unramified pro-$p$-extension of $\tilde{F}(\mu_{p}).$ More precisely, we show that under a decomposition condition, the pseudo-nullity of the $\tilde{\Lambda}$-module $X(\tilde{F}(\mu_p))(i-1)^\Delta$ is implied by the existence of a $\Z_{p}^d$-extension $L$ with $\XS{-i}{L}:=H_1(G_S(L), \Z_p(-i))$ being without torsion over the Iwasawa algebra associated to $L,$ and which contains a $\Z_{p}$-extension $F_{\infty}$ satisfying $H^2(G_{S}(F_{\infty}),\Q_{p}/\Z_p(i))=0.$ As a consequence we obtain a sufficient condition for the validity of Greenberg's generalized conjecture when the integer $i\equiv 1 \mod{[F(\mu_p):F]}.$ This existence is fulfilled for $(p, i)$-regular fields. 
	}
\end{abstract}

\hrulefill



\section{Introduction}
Let $F$ be a number field and $p$ an odd prime. The notation $S$ stands for the set of  $p$-adic and infinite primes of $F.$ Let $F_S$ be the maximal algebraic extension over $F$ which is unramified outside $S,$ and let $G_S(F)$ be the Galois group $\Gal{F_S/F}.$ For any $\Z_{p}^d$-extension $L$ of $F$ ($d\geq 1$),  we write $\Gamma_L$ for its Galois group and $\Lambda_L=\Z_p[[\Gamma_L]]$ the associated Iwasawa algebra. As usual $\tilde{F}$ will be  the compositum of all $\Z_p$-extensions of $F.$ Let $\tilde{\Gamma}$  be the Galois group of the extension $\tilde{F}/F$ and $\tilde{\Lambda}$ the associated Iwasawa algebra. It is well known that all $\Z_p$-extensions are unramified outside the set $S,$ thus $F_S$ contains $L.$ We denote by $G_S(L)$ the Galois group of the extension $F_S/L.$ Let $\XSp{L}$ be the Galois group of the maximal abelian pro-$p$-extension of $L$ which is unramified outside $S.$ In terms of homological groups we have $\XSp{L}=H_1(G_S(L), \Z_p).$ We define for any integer $i$, the twisted analogue of $\XSp{L}$ by $\XS{-i}{L}:=H_1(G_S(L), \Z_p(-i)).$ The group $\XS{-i}{L}$ has a structure of $\Lambda_L$-module, we will call it the twisted $S$-ramified Iwasawa module associated to $L.$ The aim of this paper is the study of the $\Lambda_{L}$-module $\XS{-i}{L}.$ Following Greenberg's (\cite{GC}) methods we determine its $\Lambda_{L}$-rank (Theorem \ref{the lambda rank multiple}) and we show that the projective dimension of $\XS{-i}{L}$ is at most one when $i\neq 1$, $i\not \equiv 0 \mod{[F(\mu_p): F]}$ and $F$ satisfies the twisted Leopoldt's conjecture $LC_{i}$ (Theorem \ref{multiple projective dimension}). We are interested also in the vanishing of $\Tor{\Lambda_{L}}{\XS{-i}{L}}$, the $\Lambda_{L}$-torsion sub-module of $\XS{-i}{L}$. Precisely, we give a going up property of the torsion triviality among a tower of multiple $\Z_{p}$-extensions (Theorem \ref{theorem torsion free}). As a consequence, we obtain that the vanishing of $\Tor{\tilde{\Lambda}}{\XS{-i}{\tilde{F}}}$ can be detected earlier in a $\Z_{p}$-extension of $F$ (Corollary \ref{earlier torsion}).\\
The motivation behind this study is the link between $\Tor{\tilde{\Lambda}}{\XS{-i}{\tilde{F}}}$ and a certain "eigenpart" of Greenberg's generalized conjecture.
\begin{conj}[GGC for short]
	For any number field $F$, the $\tilde{\Lambda}$-module $X(\tilde{F})$ is pseudo-null, where $X(\tilde{F})$ is the Galois group over $\tilde{F}$ of the maximal abelian unramified pro-$p$-extension of $\tilde{F}.$
\end{conj}
\noindent In \cite{Equi GGC}, the authors show, under some conditions, the equivalence between the triviality of $\Tor{\tilde{\Lambda}}{\XS{-1}{\tilde{F}}}$ and the validity of GGC. Here we prove a twisted generalization of this equivalence (Theorem \ref{formulation Greenberg conjecture}). Supposing a decomposition condition (see (Dec) below), we show for any integer $i\not \equiv 0 \mod{[F(\mu_p):F]}$ the equivalence between the following assertions:
\begin{itemize}
	\item $X(\tilde{F}(\mu_p))(i-1)^\Delta$ is pseudo-null, where $\Delta:=\Gal{\tilde{F}(\mu_{p})/\tilde{F}}$.
	\item  $\XS{-i}{\tilde{F}}$ is without $\tilde{\Lambda}$-torsion.
\end{itemize}
Summarizing, we obtain a sufficient condition for the pseudo-nullity of $X(\tilde{F}(\mu_p))(i-1)^\Delta$ in terms of the torsion sub-module of the twisted $S$-ramified Iwasawa module associated to a multiple $\Z_{p}$-extension of $F.$ More precisely, 
\begin{theo*}[Theorem \ref{main theorem}]
	Assume that the condition $(Dec)$ holds for $F.$ Let $i$ be an integer such that $i\not\equiv 0 \mod{[F(\mu_p):F]}.$ Suppose that $F$ admits  a $\Z_{p}^d$-extension $L$ such that $\Tor{\Lambda_L}{\XS{-i}{L}}=0$ and $L$ contains a $\Z_{p}$-extension $F_\infty$ such that $H^2(G_S(F_\infty), \Q_p/\Z_p(i))=0.$ Then the $\tilde{\Lambda}$-module $X(\tilde{F}(\mu_p))(i-1)^\Delta$ is pseudo-null. 
\end{theo*} 
Many theoretical results have been proved for GGC. For example, we cite the work of Minardi (\cite{Minardi}) on imaginary quadratic fields, whose results have been generalized by Itoh (\cite{Itoh}), Fujii (\cite{Fujii GGC}) and recently by Nguyen Quang Do (\cite{Nguyen19}). There methods use the ascending of the pseudo-nullity of the unramified Iwasawa module among a tower of multiple $\Z_{p}$-extensions which starts from a $\Z_{p}^{2}$-extension. The approach of this paper is based on the study of the torsion triviality of the twisted $S$-ramified Iwasawa module by an induction process which can starts from a $\Z_{p}$-extension.

The end of this paper is devoted to $(p, i)$-regular number fields. We prove that the $\tilde{\Lambda}$-module $X(\tilde{F}(\mu_p))(i-1)^\Delta$ is pseudo-null for any $(p, i)$-regular number field $F$ wich satisfies the condition $(Dec)$ (Theorem \ref{a part of GGC}). In particular, Greenberg's generalized conjecture is true for all $(p, 1)$-regular fields which verify $(Dec)$ (Corollary \ref{regualr number field}). This last result was also obtained by Nguyen Quang Do (loc. cit.).\\

\paragraph{\textbf{Notations}}
Let $F$ be a number field, and let $p$ be an odd prime number. We fix the following notations:

\begin{tabbing}
	$\Gal{N/K}$    \= the Galois group of an arbitrary Galois extension $N/K.$\\
	$L$                \> a $\Z_{p}^d$-extension of $F,$ $d\geq 1.$\\ 
	$\Gamma_L$ \>  the Galois group $\Gal{L/F}.$\\
	$\Lambda_L$ \> the associated Iwasawa algebra to $\Gamma_L,$ $\Lambda_L=\Z_p[[\Gamma_L]].$ \\
	$F^c$ \> the cyclotomic $\Z_{p}$-extension of $F.$\\
	$\Gamma_c$ \> the Galois group $\Gal{F^c/F}.$ \\
	$\Lambda_c$ \> the associated Iwasawa algebra to $\Gamma_c.$ \\
	$F_\infty$ \> an arbitrary $\Z_{p}$-extension of $F.$\\
	$\Gamma_\infty$ \> the Galois group $\Gal{F_\infty/F}.$ \\
	$\Lambda_\infty$ \> the associated Iwasawa algebra to $\Gamma_\infty.$ \\
	$S$ \>  the set of $p$-adic and infinite primes of $F$.\\
	$A(F)$ \> the $p$-primary part of the ideal class group of $F.$\\
	$A'(F)$ \> the $p$-primary part of the $(p)$-ideal class group of $F.$\\
	$X(F)$ \> the Galois group over $F$ of the maximal abelian unramified $p$-extension of $F.$\\
	$X'(F)$ \> the Galois group over $F$ of the maximal abelian unramified $p$-extension of $F$  \\
      	\> in which all $p$-primes split completely.\\
	$F_S$      \> the maximal extension of $F$ which is unramified outside $S.$\\
	$G_S(F)$   \>  the Galois group $\Gal{F_S/F}.$ \\
	$G_S(L)$ \> the Galois group $\Gal{F_S/L}.$\\
	$\XS{i}{L}$ \> the first homological group $H_1(G_S(L), \Z_p(i)).$\\
	$(\cdot)^\vee$ \> the Pontryagin dual.\\
	$\ind{\cdot}{\cdot} $ \>  the induced module.\\
	$\coind{\cdot}{\cdot}$  \>  the coinduced module.\\
	$r_1 := r_1(F)$   \> the number of real places of $F$.\\
	$r_2 := r_2(F)$ \> the number of complex places of $F$.\\
	$n_i$ \> $=\begin{cases}
	r_2 \qquad \:\:\mbox{ if } i \mbox{ is even}, \\
	r_1+r_2 \:\: \mbox{ if } i \mbox{ is odd.}
	\end{cases}
	$
\end{tabbing}
For a module $M$ over a ring $R,$ we denote by $\Tor{R}{M}$ the $R$-torsion sub-module of $M$ and $\Fr{R}{M}$ the maximal quotient of $M$ which is torsion free. We write $\mathrm{pd}_{R}M$ for the projective dimension of M over $R.$ If $R$ is a domain, let $K_R$ be its field of fractions. We call the rank of $M$ over $R$ the dimension of the vector space $M\otimes K_R,$ and we denote it by $\rank{R}M.$ Recall that a finitely generated $R$-module $M$ is called pseudo-null if $M_{\mathfrak{p}}=0$ for all prime ideals $\mathfrak{p}$ of height $\mathrm{ht}(\mathfrak{p})\leq 1,$ or equivalently any prime ideal containing the annihilator $\mathrm{Ann}_{R}(M)$ is of height $\geq 2.$ We say that two $R$-modules $M$ and $M'$ are pseudo-isomorphic (and we denote $M\sim M'$) if there exists an $R$-homomorphism with pseudo-null kernel and cokernel. 
\section{The rank of the twisted $S$-ramified Iwasawa module}
Let $\tilde{F}$ be the compositum of all $\Z_p$-extensions of $F.$ Class field theory gives a description of the Galois group $\tilde{\Gamma}=\Gal{\tilde{F}/F},$ namely 
\begin{equation*}
	\tilde{\Gamma} \simeq \Z_{p}^{r_2+1+\delta_F},
\end{equation*}
where $\delta_F$ is the Leopoldt defect, conjecturally null (Leopoldt's conjecture). A cohomological version of Leopoldt's conjecture is the triviality of the second cohomology group $H^2(G_S(F), \Q_p/\Z_p)$ (c.f \cite[Proposition $12$]{Nguyen82}). In general, we have the following twisted analogue of Leopoldt's conjecture (Greenberg, Schneider, ...)
\begin{conj}[$LC_i$ for short]
	For any number field $F$ and all integers $i\neq 1,$ the Galois cohomology group $H^2(G_S(F), \Q_p/\Z_p(i))$ vanishes.
\end{conj}  
\noindent For $i=0$ (i.e. Leopoldt's conjecture), the conjecture is true if $F/\Q$ is an abelian extension. Also, the conjectures $LC_i$ hold for any number field and any integer $i \geq 2$ by Soul\'{e} (c.f \cite{Soule}). Notice that the conjectures $LC_i$ concerns number fields, and as for the case $i=0,$ we have a weak version of $LC_i$ for an arbitrary $\Z^{d}_p$-extension $L$ of $F.$
\begin{conj}[$WLC_i$ for short]\label{WLCi}
	Let $L$ be a $\Z^{d}_p$-extension of a number field $F$, and let $i$ be an integer. If $i\neq 1$or $F^{c}\subset L$, then $H^2(G_S(L), \Q_p/\Z_p(i))=0.$
\end{conj}
\begin{rema}
If $i=1$ and $F^c \not\subset L,$ the cohomology group $H^2(G_S(L), \Q_p/\Z_p(i))$ does not vanish in general. Trough this paper, $WLC_i$ for $L$ means that $i\neq 1$ or $F^c \subset L.$
\end{rema}
\begin{prop}
	Let $L$ be a $\Z^{d}_p$-extension ($d\geq 2$) of $F,$ and let $i$ be an integer. If $L$ contains a $\Z^{d-1}_p$-extension $L'$ such that $H^2(G_S(L'),\Q_p/\Z_{p}(i))=0,$ then $H^2(G_S(L),\Q_p/\Z_{p}(i))=0.$
\end{prop}
\begin{proof}
	Denote by $\Gamma$ the Galois group of the extension $L/L'.$ Since  $\mathrm{cd}(\Gamma)\leq 1,$ the Hochschild-Serre spectral sequence
	\begin{equation*}
	\xymatrix@=2pc{ H^p(\Gamma, H^q(G_S(L), \Q_p/\Z_p(i))) \ar@{=>}[r]&  H^{p+q}(G_S(L'),  \Q_p/\Z_p(i)) }
	\end{equation*}
	induces the exact sequence
	\begin{equation*}
	\xymatrix@=1.2pc{ 0 \ar[r]& H^1(\Gamma, H^1(G_S(L), \Q_p/\Z_p(i)))\ar[r]& H^2(G_S(L'),  \Q_p/\Z_p(i)) \ar[r]	& H^2(G_S(L),  \Q_p/\Z_p(i))^{\Gamma}  \ar[r] & 0. }
	\end{equation*}
	Therefore, the vanishing of the group $H^2(G_S(L'),  \Q_p/\Z_p(i))$ implies that of $H^2(G_S(L), \Q_p/\Z_p(i)).$
\end{proof}
\begin{rema}\label{Leopoldt tordue}
	\begin{enumerate}
		\item Using the same arguments as in the above proof and the Hochschild-Serre spectral sequence associated to the following group extension 
		\begin{equation*}
		\xymatrix@=2pc{ 0 \ar[r] & G_S(F_\infty) \ar[r] & G_S(F) \ar[r]&  \Gamma_\infty \ar[r]& 0, }
		\end{equation*}
		we can prove that $WLC_i$ holds for any $\Z_p$-extension if $F$ satisfies $LC_i.$
		\item It is well known that, for any integer $i$ the cyclotomic $\Z_p$-extension satisfies $WLC_i.$ Indeed, the Galois group $G_S(F^c)$ acts trivially on $\Q_p/\Z_p(k)$ if $k\equiv 0 \mod{[F(\mu_p): F]}.$ Hence for any integer $i,$ if we choose $j\geq 2$ such that $i\equiv j \mod{[F(\mu_p): F]},$ we have 
		\begin{equation*}
		H^2(G_S(F^c),  \Q_p/\Z_p(i))=H^2(G_S(F^c),  \Q_p/\Z_p(j))(i-j).
		\end{equation*}
		Hence we conclude that the conjecture $WLC_i$ holds for the cyclotomic $\Z_p$-extension at any integer $i,$ and thus for all $\Z_{p}^d$-extension containing $F^{c}.$
	\end{enumerate}
\end{rema}
In the rest of this section we study the $\Lambda_L$-module $\XS{-i}{L},$ using the methods of Greenberg (\cite{GC}). We show that $\XS{-i}{L}$ is a finitely generated $\Lambda_L$-module and we determine its $\Lambda_L$-rank. Instead of Leopoldt's conjecture used by Greenberg, here we use $LC_i$ or $WLC_i.$ Let's start with the following helpful lemma.
\begin{lemm}\label{torsion}
	Let $L$ be a $\Z_{p}^d$-extension ($d\geq 1$) of $F,$ and let $L'$ be a subfield of $L$ such that $\Gamma := \Gal{L/L'}\simeq \Z_p.$ Then we have the following exact sequence of $\Z_p[[\Gal{L'/F}]]$-modules
	\begin{equation} \label{Exacte sequece of lemma}
	\xymatrix@=2pc{
		0 \ar[r]& (\XS{-i}{L})_{_{\Gamma}} \ar[r]& \XS{-i}{L'} \ar[r]& H_1(\Gamma, \Z_p(-i)_{G_S(L)}) \ar[r]& 0.
	}
	\end{equation}
\end{lemm}
\begin{proof}
	The  Hochschild-Serre spectral sequence associated to the group extension 
	\begin{equation*}
	\xymatrix@=2pc{
		0 \ar[r] & G_S(L) \ar[r] & G_S(L') \ar[r]& \Gamma \ar[r] & 0,
	}
	\end{equation*}
	gives rise to the following exact sequence
	\begin{equation*}
	\xymatrix@=1pc{
		0 \ar[r] & H^1(\Gamma, \Q_p/\Z_p(i)^{G_S(L)}) \ar[r] & H^1(G_S(L'), \Q_p/\Z_p(i)) \ar[r] & H^1(G_S(L), \Q_p/\Z_p(i))^{\Gamma} \ar[r] &  0,
	}
	\end{equation*}
	since $\mathrm{cd} (\Gamma)\leq 1.$ We get the exact sequence of the lemma by applying Pontryagin  duality. 
\end{proof}
Let $F_\infty$ be a $\Z_p$-extension of $F$ and $F_n$ be the $n$-th layers of $F_\infty/F.$ Using \cite[Satz $6$, page $192$]{Schneider} we have for any $i\neq 0$
\begin{equation}\label{nth rank}
	\rank{\Z_{p}} \XS{-i}{F_n}= n_ip^n + \delta_{i,n},
\end{equation}
where $\delta_{i,n}=\rank{\Z_p} H_2(G_S(F_n), \Z_p(-i))$ and
 	\begin{equation*}
 n_i= \begin{cases}
 r_2 \qquad \:\:\mbox{ if } i \mbox{ is even}, \\
 r_1+r_2 \:\: \mbox{ if } i \mbox{ is odd.}
 \end{cases}
 \end{equation*}
\begin{prop}\label{the lambda rank}
	Let $i$ be an integer, and let $F_\infty$ be a $\Z_p$-extension of $F$ satisfying  $WLC_i.$ Then 
	\begin{equation*}
		\rank{\Lambda_\infty }\XS{-i}{F_\infty}= \begin{cases}
		r_2 \qquad \:\:\mbox{ if } i \mbox{ is even}, \\
		r_1+r_2 \:\: \mbox{ if } i \mbox{ is odd.}
		\end{cases}
	\end{equation*}
\end{prop}
\begin{proof}
   Notice that the $\Z_{p}$-module $\Z_p(-i)_{G_S(F_\infty)}$ is not finite if and only if one of the following two cases holds:
	\begin{enumerate}
		\item $i = 0.$
		\item $F_\infty$ is the cyclotomic $\Z_p$-extension and $i\equiv 0 \mod{[F(\mu_p): F]}.$
	\end{enumerate}
	 The case $(1)$ is Proposition $1$ of \cite{GC}. In the case $(2),$ the group $G_S(F_\infty)$ acts trivially on $\Z_{p}(i),$ hence 
	 \begin{equation*}
	 	\XS{-i}{F_\infty}=\XSp{F_\infty}(-i).
	 \end{equation*}
	 Then one deduce the proposition by using \cite[Proposition $1$]{GC}. In the sequel we assume that we are out of these two cases. Let $\Gamma_{\infty, n}=\Gal{F_\infty/F_n}.$ The exact sequence (\ref{Exacte sequece of lemma}) and the equality (\ref{nth rank}) give that 
	\begin{equation}\label{first rank}
	\rank{\Z_p} \XS{-i}{F_\infty}_{\Gamma_{\infty, n}}= \rank{\Z_p} \XS{-i}{F_n}= n_ip^n + \delta_{i,n},
	\end{equation} 
	where $\delta_{i,n}=\rank{\Z_p} H_2(G_S(F_n), \Z_p(-i)).$ In particular, this shows that $\XS{-i}{F_\infty}$ is a finitely generated $\Lambda_\infty$-module. Let  $\rho$ be the $\Lambda_\infty$-rank of $\XS{-i}{F_\infty}.$ By the structure theorem of the $\Lambda_\infty$-module $\Fr{\Lambda_\infty}{\XS{-i}{F_\infty}}$ we have an exact sequence
	\begin{equation}\label{structure theorem}
	\xymatrix@=2pc{0 \ar[r] &\Fr{\Lambda_\infty}{\XS{-i}{F_\infty}} \ar[r]& \Lambda_{\infty}^\rho \ar[r] & Z \ar[r]& 0,		
	}
	\end{equation}
	where $Z$ is a finite $\Lambda_\infty$-module. Since $(\Lambda_{\infty}^\rho)^{\Gamma_{\infty, n}}=0,$ by the snake lemma, the exact sequence (\ref{structure theorem}) shows that 
	\begin{equation}\label{invariant of free is zero}
		\Fr{\Lambda_\infty}{\XS{-i}{F_\infty}}^{\Gamma_{\infty, n}} =0.
	\end{equation}
	The exact sequence 
	\begin{equation*}
	\xymatrix@=2pc{0 \ar[r] & \Tor{\Lambda_\infty}{\XS{-i}{F_\infty}} \ar[r]& \XS{-i}{F_\infty} \ar[r] & \Fr{\Lambda_\infty}{\XS{-i}{F_\infty}} \ar[r]& 0		
	}
	\end{equation*}
	and the snake lemma lead to the exact sequence 
	\begin{equation}\label{snake lemma}
	\xymatrix@=1pc{0 \ar[r] & \Tor{\Lambda_\infty}{\XS{-i}{F_\infty}}^{\Gamma_{\infty, n}} \ar[r]& \XS{-i}{F_\infty}^{\Gamma_{\infty, n}} \ar[r] & \: \Fr{\Lambda_\infty}{\XS{-i}{F_\infty}}^{\Gamma_{\infty, n}} \:	\ar@{-<}  `[d]   `[d] [d] \\ 
	0	& \Fr{\Lambda_\infty}{\XS{-i}{F_\infty}}_{\Gamma_{\infty, n}} \ar[l] & \XS{-i}{F_\infty}_{\Gamma_{\infty, n}} \ar[l] & \Tor{\Lambda_\infty}{\XS{-i}{F_\infty}}_{\Gamma_{\infty, n}} \ar[l]		
	}
	\end{equation}
	According to (\ref{invariant of free is zero}), the  exact sequence (\ref{snake lemma}) gives the exact sequence
	\begin{equation}\label{rank exact sequence}
	\xymatrix@=1pc{0 \ar[r] & \Tor{\Lambda_\infty}{\XS{-i}{F_\infty}}_{\Gamma_{\infty, n}} \ar[r]& \XS{-i}{F_\infty}_{\Gamma_{\infty, n}} \ar[r] & \Fr{\Lambda_\infty}{\XS{-i}{F_\infty}}_{\Gamma_{\infty, n}} \ar[r]& 0.		
	}
	\end{equation}
	and the isomorphism
	\begin{equation}\label{isomorphism invariant}
	(\Tor{\Lambda_\infty}{\XS{-i}{F_\infty}})^{\Gamma_{\infty, n}} \simeq \XS{-i}{F_\infty}^{\Gamma_{\infty, n}}.
	\end{equation}
	It is well known that (e.g \cite[page $337$]{Wingberg})
	\begin{equation}\label{isomorphism rank}
	\rank{\Z_p} (\Tor{\Lambda_\infty}{\XS{-i}{F_\infty}})^{\Gamma_{\infty, n}} = \rank{\Z_p} (\Tor{\Lambda_\infty}{\XS{-i}{F_\infty})}_{\Gamma_{\infty, n}}. 
	\end{equation}
	Since $H^2(G_S(F_\infty), \Q_p/\Z_p(i))=0$ (by $WLC_i$) and $cd(\Gamma_{\infty, n})\leq 1,$ the Hochschild-Serre spectral sequence associated to the group extension
	\begin{equation*}
	\xymatrix@=2pc{ 0 \ar[r] & G_S(F_\infty) \ar[r]& G_S(F_n) \ar[r] & \Gamma_{\infty, n} \ar[r]& 0,		
	}
	\end{equation*}
	gives the following isomorphism 
	\begin{equation*}
	H^2(G_S(F_n), \Q_p/\Z_p(i)) \simeq H^1(\Gamma_{\infty, n}, H^1(G_S(F_\infty), \Q_p/\Z_{p}(i))).
	\end{equation*}
	Then by the cohomology-homology duality we get 
	\begin{equation*}
	H_2(G_S(F_n), \Z_p(-i)) \simeq \XS{-i}{F_\infty}^{\Gamma_{\infty, n}}.
	\end{equation*}
	From (\ref{isomorphism invariant}) and (\ref{isomorphism rank}) it follows that 
	\begin{equation}\label{Leopoldt and torsion}
	\rank{\Z_p}(\Tor{\Lambda_\infty}{\XS{-i}{F_\infty}})_{\Gamma_{\infty, n}}= \delta_{i,n}.
	\end{equation}
	The $\Gamma_{\infty, n}$-homology of the exact sequence (\ref{structure theorem}) gives rise to the following exact sequence
		\begin{equation*}
	\xymatrix@=2pc{0 \ar[r] & Z^{\Gamma_{\infty, n}} \ar[r]& \Fr{\Lambda_\infty}{\XS{-i}{F_\infty}}_{\Gamma_{\infty, n}} \ar[r]& (\Lambda_{\infty}^\rho)_{\Gamma_{\infty, n}} \ar[r] & Z_{\Gamma_{\infty, n}} \ar[r]& 0.		
	}
	\end{equation*}
	Also we know that the $\Z_p$-rank of $(\Lambda_{\infty})_{\Gamma_{\infty, n}}$ is $p^n,$ since 
	\begin{equation*}
	(\Lambda_{\infty})_{\Gamma_{\infty, n}} \simeq \Z_p[[\Gamma_\infty/\Gamma_{\infty, n}]].
	\end{equation*}
	Then, the $\Z_p$-rank of $\Fr{\Lambda_\infty}{\XS{-i}{F_\infty}}_{\Gamma_{\infty, n}}$ equals to $\rho p^n.$\\
	According to (\ref{first rank}), by the exact sequence (\ref{rank exact sequence})  and the equality (\ref{Leopoldt and torsion}) we obtain that 
	\begin{equation*}
		\rho = n_i.
	\end{equation*}
\end{proof}
The following result is an extension of the above theorem to multiple $\Z_p$-extensions.
\begin{theo}\label{the lambda rank multiple}
	Let $L$ be a $\Z_{p}^d$-extension ($d \geq 1$) of $F,$ which contains a $\Z_p$-extension $F_\infty$ satisfying $WLC_i.$ Then 
	\begin{equation*}
	\rank{\Lambda_L }\XS{-i}{L}= \begin{cases}
	r_2 \qquad \:\:\mbox{ if } i \mbox{ is even}, \\
	r_1+r_2 \:\: \mbox{ if } i \mbox{ is odd.}
	\end{cases}
	\end{equation*}
\end{theo}
\begin{proof}
	We proceed by induction on $d.$ For $d=1,$ this is Proposition \ref{the lambda rank}.\\ Let $d\geq 2$ and $L'\subset L$ be a $\Z_{p}^{d-1}$-extension of $F$ satisfying 
	\begin{itemize}
		\item $L'$ contains $F_\infty$,
		\item the Galois group $\Gamma=\Gal{L/L'}$ is topologically generated by an element $\gamma \in \Gamma$ such that $\gamma -1$ is prime to the annihilator of $\Tor{\Lambda_L}{\XS{-i}{L}}.$ 
	\end{itemize}
     Such a choice is possible since the annihilator of $\Tor{\Lambda_L}{\XS{-i}{L}}$ is divisible by a finite number of primes (e.g \cite[Lemma 2]{GC}).
	 Let $\Gamma_{L'}=\Gal{L'/F}$ and $\Lambda_{L'}$ the associated Iwasawa algebra to $\Gamma_{L'}.$ This choice of $L'$ implies that $\Tor{\Lambda_L}{\XS{-i}{L}}_{\Gamma}$ is $\Lambda_{L'}$-torsion.\\
	Let $x$ be an element of $\Fr{\Lambda_L}{\XS{-i}{L}}.$ We have
	\begin{equation*}
	x \in \Fr{\Lambda_L}{\XS{-i}{L}}^{\Gamma} \mbox{ if and only if } (\gamma -1)x=0.
	\end{equation*}
	Since $\gamma -1$ is an element of $\Lambda_L$, $\Fr{\Lambda_L}{\XS{-i}{L}}$ is without $\Lambda_L$-torsion and $H_1(\Gamma, \Fr{\Lambda_L}{\XS{-i}{L}})=\Fr{\Lambda_L}{\XS{-i}{L}}^{\Gamma}$, we obtain that
	\begin{equation*}
	H_1(\Gamma, \Fr{\Lambda_L}{\XS{-i}{L}})=0.
	\end{equation*}
	Therefore, the $\Gamma$-homology of the exact sequence
	\begin{equation*}
	\xymatrix@=2pc{0 \ar[r] & \Tor{\Lambda_L}{\XS{-i}{L}} \ar[r]& \XS{-i}{L} \ar[r] & \Fr{\Lambda_L}{\XS{-i}{L}} \ar[r]& 0		}
	\end{equation*}
	gives rise to the exact sequence
	\begin{equation}\label{free quotient exact sequence}
	\xymatrix@=2pc{0 \ar[r] & \Tor{\Lambda_L}{\XS{-i}{L}}_{\Gamma} \ar[r]& \XS{-i}{L}_{\Gamma} \ar[r] & \Fr{\Lambda_L}{\XS{-i}{L}}_{\Gamma} \ar[r]& 0.		}
	\end{equation}
	Recall that $\Tor{\Lambda_L}{\XS{-i}{L}}_{\Gamma}$ is $\Lambda_{L'}$-torsion, hence it follows that
	\begin{eqnarray}
		\rank{\Lambda_{L'}}\XS{-i}{L}_{\Gamma}&=& \rank{\Lambda_{L'}} \Fr{\Lambda_L}{\XS{-i}{L}}_{\Gamma} \nonumber \\
		&=& \rank{\Lambda_L} \XS{-i}{L}, \label{rank 1}
	\end{eqnarray}
	since $\Fr{\Lambda_L}{\XS{-i}{L}}$ is without $\Lambda_L$-torsion and has the same $\Lambda_L$-rank as $\XS{-i}{L}$ (e.g \cite[page $90$]{GC}).\\
	The exact sequence (\ref{Exacte sequece of lemma}) shows that 
	\begin{equation}\label{rank 2}
		\rank{\Lambda_{L'}} \XS{-i}{L'} = \rank{\Lambda_{L'}} \XS{-i}{L}_{\Gamma} + \rank{\Lambda_{L'}} H_1(\Gamma, \Z_p(-i)_{G_S(L)}).
	\end{equation}
	Observe that $H_1(\Gamma, \Z_p(-i)_{G_S(L)})= (\Z_p(-i)_{G_S(L)})^{\Gamma},$ thus $H_1(\Gamma, \Z_p(-i)_{G_S(L)})$ is either finite or isomorphic to $\Z_p.$ In particular, $\rank{\Lambda_{L'}} H_1(\Gamma, \Z_p(-i)_{G_S(L)})$ is zero.
	Therefore, it follows from (\ref{rank 1}) and (\ref{rank 2}) that
	\begin{equation*}
	\rank{\Lambda_L}\XS{-i}{L}= \rank{\Lambda_{L'}} \XS{-i}{L'}.
	\end{equation*}
\end{proof}

\section{Projective dimension and a going up of torsion triviality}
The aim of this section is the study of the $S$-ramified Iwasawa modules in the view of the projective dimension and torsion sub-module. Precisely, we show that its projective dimension is at most one. Also we prove that, under $WLC_{i}$, the vanishing of its torsion goes up trough a tower of multiple $\Z_{p}$-extensions. 

\paragraph{\textbf{Projective dimension}}
Now we focus on the projective dimension of the twisted $S$-ramified Iwasawa module $\XS{-i}{L}.$ As in the previous section, we first examine the case where $L$ is a $\Z_p$-extension of $F.$
\begin{prop}\label{proj dim}
	Let $F_\infty$ be a $\Z_p$-extension of $F$ for which $WLC_i$ holds. Then   $\XS{-i}{F_{\infty}}$ is of projective dimension at most $1.$
\end{prop}
\begin{proof}
	It is well known that $\XS{-i}{F_{\infty}}$ is of projective dimension at most $1$ if and only if $\XS{-i}{F_{\infty}}^{\Gamma_{\infty, n}}$ is $\Z_p$-free for some $n,$ where $\Gamma_{\infty, n}=\Gal{F_\infty/F_n}$ and $F_n$ is the $n$th layer of $F_\infty,$ \cite[Proposition $2.1$]{Wingberg}. We claim that $\XS{-i}{F_{\infty}}^{\Gamma_\infty}$ is $\Z_p$-free. \\
	Since $\mathrm{cd} (\Gamma_\infty)\leq 1,$ the homological Hoschilde-Serre spectral sequence  associated to the group extension
	\begin{equation*}
	\xymatrix@=2pc{0 \ar[r]& G_S(F_\infty) \ar[r] & G_S(F) \ar[r] & \Gamma_\infty  \ar[r]& 0}
	\end{equation*} 
	gives the following isomorphism 
	\begin{equation}\label{isomorphisim proj}
	H_1(\Gamma_\infty, \XS{-i}{F_{\infty}}) \simeq H_2(G_S(F), \Z_p(-i)).
	\end{equation}
	We know that $\mathrm{cd}(G_S(F)) \leq 2,$ so the $G_S(F)$-cohomology of the exact sequence 
	\begin{equation*}
	\xymatrix@=2pc{ 0 \ar[r] & \Z_p(i) \ar[r] & \Q_p(i) \ar[r]& \Q_p/\Z_p(i) \ar[r]& 0}
	\end{equation*}
	shows that $H^2(G_S(F), \Q_p/\Z_p(i))$ is divisible. Hence $H_2(G_S(F), \Z_p(-i))$ is a free $\Z_p$-module. Recall that 
	\begin{equation*}
	\XS{-i}{F_{\infty}}^{\Gamma_\infty} = H_1(\Gamma_\infty, \XS{-i}{F_{\infty}}).
	\end{equation*}
	It follows then from the isomorphism (\ref{isomorphisim proj}) that $\XS{-i}{F_{\infty}}^{\Gamma_\infty}$ is  $\Z_{p}$-free which means that $\XS{-i}{F_{\infty}}$ is of projective dimension at most $1.$
\end{proof}
For any module $M$ over $\Z_{p}[[\Gamma]],$ with $\Gamma \simeq \Z_{p},$ we have an equivalence between $\mathrm{pd}_{\Z_{p}[[\Gamma]]} M\leq 1$ and $M$ has no non trivial pseudo-null sub-module. However, this equivalence is not in general true when $\Gamma\simeq \Z_{p}^d$ with $d\geq 2$ (c.f \cite[Proposition $6$ chap. I]{Perin Riou}, see also $(2)$ of Remark \ref{remark proj dim}). To extend Proposition \ref{proj dim} to multiple $\Z_p$-extensions, we start with the study of pseudo-null sub-modules. For $i=0$ we have the following result of Greenberg.
\begin{prop}[\cite{GC}, Proposition $5$] \label{Greenberg's proposition}
	Assume that Leopoldt's conjecture is valid for $F.$ Let $L$ be a $\Z_{p}^d$-extension of $F$ ($d\geq 1$). Then $\XSp{L}$ contains no non trivial pseudo-null $\Lambda_L$-sub-module.
\end{prop}
The next result is a generalization of Proposition \ref{Greenberg's proposition}.
\begin{prop}\label{pseudo null sub-modules}
	Let $i\neq 1$ be an integer and suppose that $F$ satisfies $LC_i.$ Let $L$ be a $\Z_{p}^d$-extension ($d\geq 1$)  of $F.$ Then $\XS{-i}{L}$ contains no non trivial pseudo-null sub-module.
\end{prop}
\begin{proof}(Sketch)
	The proof of this proposition is an adaptation of that of \cite[Proposition $5$]{GC}. It suffices to take $\XS{-i}{L}$ instead of $\XSp{L}$ and remark that the restriction map 
	\begin{equation*}
	\xymatrix@=2pc{
		(\XS{-i}{L})_{\Gamma} \ar[r]& \XS{-i}{L'},
	}
	\end{equation*}
	is injective for any integer $i$ and any $\Z_{p}^{d-1}$-subextension $L'$ of $L$(see Lemma \ref{torsion}). Also we replace Leopoldt's conjecture by $LC_i.$
\end{proof}
Now we can prove that the projective dimension of $\XS{-i}{L}$ over $\Lambda_{L}$ is at most one when $i\neq 1$ and $i\not \equiv  0 \mod{[F(\mu_p):F]}.$
\begin{theo}\label{multiple projective dimension}
	Let $i \neq 1$ be an integer such that $i\not \equiv 0 \mod{[F(\mu_p):F]}$ and $F$ satisfies $LC_i.$ For any $\Z_{p}^d$-extension $L$ of $F,$ the $\Lambda_{L}$-module $\XS{-i}{L}$ is of projective dimension at most $1.$
\end{theo}
\begin{proof}
	We prove this theorem by induction on $d.$ The case $d=1$ is Proposition \ref{proj dim}. For $d\geq 2$ we choose a subgroup $\Gamma\simeq \Z_p$ of $\Gamma_L$ and denote by $L'$ the subfield of $L$ fixed by $\Gamma.$ 
	Recall that 
	\begin{equation*}
	H_0(G_S(L), \Z_p(-i))=0 \mbox{  if and only if }  i \not\equiv 0 \mod{[F(\mu_p):F]}.
	\end{equation*}
	When the integer $i$ satisfies $i\not\equiv 0 \mod{[F(\mu_p):F]},$ we obtain by Lemma \ref{torsion} the isomorphism 
	\begin{equation}\label{isomorphism 1}
	(\XS{-i}{L})_{_{\Gamma}} \simeq \XS{-i}{L'}.
	\end{equation}
	By Proposition \ref{pseudo null sub-modules}, the $\Lambda_{L}$-module $\XS{-i}{L}$ has no non trivial pseudo-null sub-module. According to the proof of \cite[Lemme $5,$ chap. I]{Perin Riou}, we have 
	\begin{equation*}
	\mathrm{pd}_{\Lambda_{L}}(\XS{-i}{L}) = \mathrm{pd}_{\Lambda_{L}/\mathfrak{q}}(\XS{-i}{L}/\mathfrak{q}\XS{-i}{L}),
	\end{equation*}
	for every prime ideal $\mathfrak{q}$ of height $1$ such that $\Lambda_{L}/\mathfrak{q}$ is a regular ring and $\mathfrak{q}$ does not divide the characteristic series of $\XS{-i}{L}.$ \\
	Notice that $(\Lambda_{L})_{\Gamma} \simeq \Lambda_{L'}$ which is a regular ring.
	Let $\gamma$ be a topological generator of $\Gamma.$ Then if we choose $\Gamma$ such that $\gamma-1$ does not divide the characteristic series of $\XS{-i}{L}$ (such a choice always exists, since the characteristic series of $\XS{-i}{L}$ is divisible by a finite number of primes) we obtain that 
	\begin{equation*}
	\mathrm{pd}_{\Lambda_{L}}(\XS{-i}{L}) = \mathrm{pd}_{\Lambda_{L'}}((\XS{-i}{L})_{\Gamma}).
	\end{equation*}  
	Therefore, it follows from (\ref{isomorphism 1}) that 
	\begin{equation*}
	\mathrm{pd}_{\Lambda_{L}}(\XS{-i}{L}) = \mathrm{pd}_{\Lambda_{L'}}(\XS{-i}{L'}).
	\end{equation*}
\end{proof}
\begin{rema} \label{remark proj dim}
	\begin{enumerate}
		\item For $i=1$ the above theorem is still true if $L$ contains $F^c$. This follows from the fact that $LC_j$ holds for $j\geq 2,$ and 
		\begin{equation*}
			\XS{-i}{L}=\XS{-j}{L}(j-i) \mbox{ for any } j\equiv i \mod{[F(\mu_p):F]}.
		\end{equation*} 
		\item When $i = 0$ and $L$ satisfies $WLC_i,$ Nguyen Quang Do (\cite[Proposition 3.4]{Nguyen83}) shows that 
		\begin{itemize}
			\item $\mathrm{pd}_{\Lambda_{L}}\XSp{L} \leq 1$ if $d\leq 2$ and,
			\item $\mathrm{pd}_{\Lambda_{L}}\XSp{L} =d-2$ if $d>2.$
		\end{itemize}	
		Notice that if $i \equiv 0 \mod{[F(\mu_p):F]}$ and $L$ contains the cyclotomic $\Z_{p}$-extension, the discussion in the proof of Proposition \ref{the lambda rank} shows that 
		\begin{equation*}
		\XS{-i}{L}=\XSp{L}(-i).
		\end{equation*}
		Then by \cite[Proposition 3.4]{Nguyen83}, we have
		\begin{itemize}
			\item $\mathrm{pd}_{\Lambda_{L}}\XS{-i}{L} \leq 1$ if $d\leq 2$ and,
			\item $\mathrm{pd}_{\Lambda_{L}}\XS{-i}{L} =d-2$ if $d>2.$
		\end{itemize}
		In particular, in this case, if we suppose also that $d>3$ we find that $\mathrm{pd}_{\Lambda_{L}}\XS{-i}{L}>1.$ This is one of the reasons why we assume that $i \not\equiv 0 \mod{[F(\mu_p):F]}.$ Another reason is that our proof uses the isomorphism \eqref{isomorphism 1}, which is not true when $i \equiv 0 \mod{[F(\mu_p):F]}.$
	\end{enumerate}
\end{rema}

\paragraph{\textbf{Going up for torsion triviality}}
Now we prove that the triviality of the $\tilde{\Lambda}$-torsion of the twisted $S$-ramified Iwasawa module of $\tilde{F}$ can be detected earlier from a $\Z_{p}$-extension of the base field $F.$ Fix a tower of $\Z_p$-extensions 
\begin{equation*}
F_\infty= F^{(1)} \subset \cdots \subset F^{(d)} \cdots \subset F^{(r)}=\tilde{F},
\end{equation*}
where $F_\infty$ is a $\Z_p$-extension of $F$ satisfying $WLC_i$ and $r=r_2+1+\delta_F.$ The choice of such a $\Z_p$-extension is always possible since the cyclotomic $\Z_p$-extension $F^{c}$ verifies $WLC_i.$ For simplicity, we will write $\X{d}{i}$ instead of $\XS{i}{F^{(d)}}$ and $\Lambda_{(d)}$ instead of $\Lambda_{F^{(d)}}.$ We start with the following helpful lemma (compare with \cite[$1.$ of Lemma $4$]{Perin Riou}).
\begin{lemm} \label{coinvariante is torsion}
	For any integer $1\leq d \leq r=r_2+1+\delta_F,$  $(\Tor{\Lambda_{(d+1)}}{\X{d+1}{-i}})_\Gamma$ is a $\Lambda_{(d)}$-torsion module, where $\Gamma$ is the Galois group of the extension $F^{(d+1)}/F^{(d)}.$
\end{lemm}
\begin{proof}
	Remark that for all $d\geq 1,$ the $\Z_{p}^d$-extension $F^{(d)}$ contains $F_\infty.$ Thus according to Theorem \ref{the lambda rank multiple} we know that $\X{d}{i}$ is of $\Lambda_{(d)}$-rank equal to $n_i.$ The exact sequence (\ref{free quotient exact sequence}) for $L=F^{(d+1)}$ and $L'=F^{(d)}$ writes
	\begin{equation*}
	\xymatrix@=1pc{0 \ar[r] & (\Tor{\Lambda_{(d+1)}}{\X{d+1}{-i}})_{\Gamma} \ar[r]& (\X{d+1}{-i})_{\Gamma} \ar[r] & (\Fr{\Lambda_{(d+1)}}{\X{d+1}{-i}})_{\Gamma} \ar[r]& 0.	}
	\end{equation*}
	Then we have 
	\begin{equation}\label{equality 1}
	\rank{\Lambda_{(d)}} (\X{d+1}{-i})_{\Gamma} =    \rank{\Lambda_{(d)}}(\Fr{\Lambda_{(d+1)}}{\X{d+1}{-i}})_{\Gamma} + \rank{\Lambda_{(d)}}(\Tor{\Lambda_{(d+1)}}{\X{d+1}{-i}})_{\Gamma}\:.
	\end{equation}
	We know that $H_1(\Gamma, \Z_p(i)_{G_S(F^{(d+1)})})$ is $\Lambda_{(d)}$-torsion, hence by Lemma \ref{torsion} we obtain 
	\begin{equation*}
	\rank{\Lambda_{(d)}}(\X{d+1}{-i})_{\Gamma} = \rank{\Lambda_{(d)}}(\X{d}{-i})=n_i=\begin{cases}
	r_2 \qquad \:\:\mbox{ if } i \mbox{ is even}, \\
	r_1+r_2 \:\: \mbox{ if } i \mbox{ is odd.}
	\end{cases}
	\end{equation*}
	Since $\rank{\Lambda_{(d)}}(\Fr{\Lambda_{(d+1)}}{\X{d+1}{-i}})_{\Gamma}= \rank{\Lambda_{(d+1)}} (\X{d+1}{-i})=n_i,$ the equality (\ref{equality 1}) implies that 
	\begin{equation*}
	\rank{\Lambda_{(d)}}(\Tor{\Lambda_{(d+1)}}{\X{d+1}{-i}})_{\Gamma}=0.
	\end{equation*}
\end{proof}
\begin{theo}\label{theorem torsion free}
	Let $F$ be a number field which admits a $\Z_p$-extension $F_\infty$ satisfying $WLC_i.$ Let $d$ be an integer such that $1 \leq d < r=r_2+1+\delta_F$ and $\Tor{\Lambda_{(d)}}{\X{d}{-i}}=0.$ Then $\Tor{\Lambda_{(d+1)}}{\X{d+1}{-i}}=0.$ \\
	In particular, if there exists $d$, $1 \leq d \leq r$ such that  $\Tor{\Lambda_{(d)}}{\X{d}{-i}}=0$, then $\X{k}{-i}$ is a $\Lambda_{(k)}$-torsion-free module for all $d \leq k \leq r.$ 
\end{theo}
\begin{proof}
	Let $  F_\infty= F^{(1)} \subset \cdots \subset F^{(d)} \cdots \subset F^{(r)}=\tilde{F}$ be a fixed tower of multiple $\Z_p$-extensions associated to the field $F.$ Let $\Gamma=\Gal{F^{(d+1)}/F^{(d)}}.$ The exact sequence 
	\begin{equation*}
	\xymatrix@=2pc{
		0 \ar[r] & \Tor{\Lambda_{(d+1)}}{\X{d+1}{-i}} \ar[r] & \X{d+1}{-i} \ar[r]& \Fr{\Lambda_{(d+1)}}{\X{d+1}{-i}} \ar[r]& 0,
	}
	\end{equation*}
	gives by the snake lemma the following exact sequence
	\begin{equation}\label{torsion sequence}
	\xymatrix@=0.6pc{
		(\Fr{\Lambda_{(d+1)}}{\X{d+1}{-i}})^\Gamma \ar[r] & (\Tor{\Lambda_{(d+1)}}{\X{d+1}{-i}})_\Gamma \ar[r] & (\X{d+1}{-i})_\Gamma \ar[r]& (\Fr{\Lambda_{(d+1)}}{\X{d+1}{-i}})_\Gamma \ar[r]& 0.
	}
	\end{equation}
	The $\Lambda_{(d+1)}$-module $\Fr{\Lambda_{(d+1)}}{\X{d+1}{-i}}$ is without torsion, hence 
	\begin{equation*}
	(\Fr{\Lambda_{(d+1)}}{\X{d+1}{-i}})^\Gamma=0.
	\end{equation*}
	Then the exact sequence (\ref{torsion sequence}) becomes 
	\begin{equation}\label{torsion triviality}
	\xymatrix@=2pc{
		0 \ar[r] & (\Tor{\Lambda_{(d+1)}}{\X{d+1}{-i}})_\Gamma \ar[r] & (\X{d+1}{-i})_\Gamma \ar[r]& (\Fr{\Lambda_{(d+1)}}{\X{d+1}{-i}})_\Gamma \ar[r]& 0.
	}
	\end{equation}
	The exact sequence (\ref{Exacte sequece of lemma}) shows that if $\X{d}{-i}$ is $\Lambda_{(d)}$-torsion free then $(\X{d+1}{-i})_\Gamma$ is also $\Lambda_{(d)}$-torsion free. Furthermore, Lemma \ref{coinvariante is torsion} tells us that the module $(\Tor{\Lambda_{(d+1)}}{\X{d+1}{-i}})_\Gamma$ is $\Lambda_{(d)}$-torsion. Therefore, the exact sequence (\ref{torsion triviality}) implies that $(\Tor{\Lambda_{(d+1)}}{\X{d+1}{-i}})_\Gamma=0$ which gives, by Nakayama's Lemma, that
	\begin{equation*} 
	\Tor{\Lambda_{(d+1)}}{\X{d+1}{-i}}=0.
	\end{equation*}  
\end{proof}
\begin{coro}\label{earlier torsion}
	Let $F$ be a number field which admits a $\Z_p$-extension $F_\infty$ satisfying $WLC_i.$ If $\Tor{\Lambda_{\infty}}{\XS{-i}{F_{\infty}}}$ is trivial, then $\XS{-i}{\tilde{F}}$ is a free torsion $\tilde{\Lambda}$-module.
\end{coro}

\section{Greenberg conjecture and torsion triviality}
For any number field $F,$ let $X(F)$ (respectively $X'(F)$) be the Galois group of the maximal abelian unramified (respectively unramified such that all $p$-primes split completely) $p$-extension of $F.$ The group $X(F)$ (resp. $X'(F)$) is isomorphic to the $p$-primary part of the ideal class group $A(F)$ (resp. $(p)$-ideal class group $A'(F)$). Let $L$ be a multiple $\Z_p$-extension of $F.$ We consider the two $\Lambda_L$-modules 
\begin{equation*}
X(L)=\underset{\leftarrow}{\lim} X(E)    \mbox{\: and \:} 	X'(L)=\underset{\leftarrow}{\lim} X'(E),
\end{equation*}
where $E$ runs over finite sub-extensions of $L$ and the inverse limit is taken via the norm maps. It is well known that $X(L)$ (resp. $X'(L)$) is isomorphic to the Galois group of the maximal abelian unramified (respectively unramified such that all $p$-primes split completely) pro-$p$-extension of $L.$ In \cite{GC76} and \cite{GGC}, Greenberg proposed respectively the following conjectures. 
\begin{conj}[GC for short]
	The $\Lambda_c$-module $X(F^c)$ is finite for any totally real number field $F.$
\end{conj}
\begin{conj}[GGC for short]
	For any number field $F,$ the $\tilde{\Lambda}$-module $X(\tilde{F})$ is pseudo-null.
\end{conj}
We find in the literature a weak version of GGC concerning the unramified $p$-decomposed Iwasawa module (\textit{e.g.} \cite{Nguyen19}).
\begin{conj}[GGC' for short]
	For any number field $F,$ the $\tilde{\Lambda}$-module $X'(\tilde{F})$ is pseudo-null.
\end{conj}
By construction, the $\tilde{\Lambda}$-module $X'(\tilde{F})$ is a quotient of $X(\tilde{F}),$ thus the validity of GGC implies that of GGC'. Moreover, the other implication was considered by Nguyen  Quang Do in \cite{Nguyen19}, precisely he shows that
\begin{equation}\label{etale equivalence}
GGC \mbox{ holds for } F \Longleftrightarrow GGC' \mbox{ holds for } F,
\end{equation} 
when the field $F$ satisfies the following condition: 
\begin{center}\label{condition Dec}
	(Dec) : \textit{All $p$-adic places has decomposition group in the extension $\tilde{F}/F$ of $\Z_p$-rank at least $2$}.
\end{center}
This condition is known to be true for any imaginary Galois extension of $\Q$ and any number field which contains $\mu_p$ (\cite[Theorem 3.2]{Equi GGC}).\\
The next result shows that the equivalence (\ref{etale equivalence}) remains true if we take the $e_i$-components of $X'(\tilde{F}(\mu_p))$ and $X(\tilde{F}(\mu_p))$
\begin{prop}\label{proposition chi composante}
	Let $F$ be a number field which satisfies the condition (Dec). For any integer $i,$ we have a pseudo-isomorphism
	\begin{equation*}
	X(\tilde{F}(\mu_p))(i)^{\Delta} \sim X'(\tilde{F}(\mu_p))(i)^{\Delta},
	\end{equation*}
	where $\Delta$ is the Galois group of the extension $\tilde{F}(\mu_p)/\tilde{F}.$
\end{prop}	
\begin{proof}
	Let $E/F$ be a finite subextension  of $\tilde{F}(\mu_p)/F.$ We have an exact sequence (e.g \cite[exact sequence ($3\mu$)]{BN})
	\begin{equation*}
	\xymatrix@=2pc{0 \ar[r]& U_{E} \ar[r]& U'_{E} \ar[r]&  \displaystyle{\bigoplus_{v\in S_p(E)}}  \Z_{p} \ar[r]&  A(E)  \ar[r]&  A'(E)   \ar[r]&  0, }
	\end{equation*}
	where $U_E$ (resp. $U'_E$)  is the group of units (resp. $(p)$-units) of $E$ and $S_p(E)$ is the set of $p$-adic places in $E.$
	Taking the projective limit over all finite subextensions of $\tilde{F}(\mu_p)/F$ with respect to the norm maps we get the exact sequence
	\begin{equation}\label{chi composante}
	\xymatrix@=1pc{0 \ar[r]& \underset{F\subset E \subset \tilde{F}(\mu_p)}{\varprojlim} U_{E} \ar[r]& \underset{F\subset E \subset \tilde{F}(\mu_p)}{\varprojlim} U'_{E} \ar[r]&  \underset{F\subset E \subset \tilde{F}(\mu_p)}{\varprojlim} \displaystyle\bigoplus_{v\in S_p(E)}  \Z_{p} \ar  `[d]   `[d] [d]\\
		&0 &  X'(\tilde{F}(\mu_p))   \ar[l]&  X(\tilde{F}(\mu_p))  \ar[l] }
	\end{equation}
	For any finite subextension $E/F$ of $\tilde{F}(\mu_p)/F$ we denote by $G_{E}$the Galois group $\Gal{E/F}$ and by $G_{E, v}$ the decomposition subgroup of $v$ in the extension $E/F.$ Let $\tilde{\Gamma}_E=\Gal{\tilde{F}(\mu_p)/E}$ and $\tilde{\Gamma}_{E, v}$ the decomposition subgroup at $v$ in the extension $\tilde{F}(\mu_p)/E.$ Observe that $G_{E} = \tilde{\Gamma}_F/\tilde{\Gamma}_{E}$ and,
	\begin{align*}
	G_{E,v} &= \tilde{\Gamma}_{F,v}/\tilde{\Gamma}_{E,v}\\
	&= \tilde{\Gamma}_{F,v}/\tilde{\Gamma}_{F,v} \cap \tilde{\Gamma}_{E}\\
	&= \tilde{\Gamma}_{F,v} \tilde{\Gamma}_{E}/\tilde{\Gamma}_{E}.
	\end{align*} 
	We have 
	\begin{align*}
	\underset{F\subset E \subset \tilde{F}(\mu_p)}{\varprojlim} \bigoplus_{v\in S_p(E)}  \Z_{p}  &= \underset{F\subset E \subset \tilde{F}(\mu_p)}{\varprojlim} \bigoplus_{v\in S_p(F)} \bigoplus_{w \mid v} \Z_{p} \\
	&= \bigoplus_{v\in S_p(F)} \underset{F\subset E \subset \tilde{F}(\mu_p)}{\varprojlim} \mathrm{Ind}^{G_{E}}_{G_{E, v}} \Z_{p}\\
	&= \bigoplus_{v\in S_p(F)} \underset{F\subset E \subset \tilde{F}(\mu_p)}{\varprojlim} \mathrm{Ind}^{\tilde{\Gamma}_F/\tilde{\Gamma}_{E}}_{\tilde{\Gamma}_{F,v} \tilde{\Gamma}_{E}/\tilde{\Gamma}_{E}} \Z_{p}
	\end{align*}
	By \cite[Theorem $6.10.8$]{Zariski} and the definition of the induced module we have  
	\begin{align*}
	\bigoplus_{v\in S_p(F)} \underset{F\subset E \subset \tilde{F}(\mu_p)}{\varprojlim} \mathrm{Ind}^{\tilde{\Gamma}_F/\tilde{\Gamma}_{E}}_{\tilde{\Gamma}_{F,v} \tilde{\Gamma}_{E}/\tilde{\Gamma}_{E}} \Z_{p} &= \bigoplus_{v\in S_p(F)} \mathrm{Ind}^{\tilde{\Gamma}_F}_{\tilde{\Gamma}_{F,v}} \Z_{p}\\
	&= \bigoplus_{v\in S_p(F)} \Z_{p} \hat{\otimes}_{\Z_{p}[[\tilde{\Gamma}_{F,v}]]} \Z_{p}[[\tilde{\Gamma}_F]]
	\end{align*}
	According to \cite[Proposition $5.8.1$]{Zariski} we get 
	\begin{equation*}
		\bigoplus_{v\in S_p(F)} \Z_{p} \hat{\otimes}_{\Z_{p}[[\tilde{\Gamma}_{F,v}]]} \Z_{p}[[\tilde{\Gamma}_F]]
		\simeq \bigoplus_{v\in S_p(F)} \Z_{p}[[\tilde{\Gamma}_F/\tilde{\Gamma}_{F,v}]].
	\end{equation*}
	Hence we conclude that 
	\begin{equation*}
	\underset{F\subset E \subset \tilde{F}(\mu_p)}{\varprojlim} \bigoplus_{v\in S_p(E)}  \Z_{p}
	\simeq \bigoplus_{v\in S_p(F)} \Z_{p}[[\tilde{\Gamma}_F/\tilde{\Gamma}_{F,v}]].
	\end{equation*}
	After twisting $i$ times and taking the $\Delta$-invariants, the exact sequence (\ref{chi composante}) becomes
		\begin{equation}\label{eq seq etale}
	\xymatrix@=1pc{0 \ar[r]& \underset{F\subset E \subset \tilde{F}(\mu_p)}{\varprojlim} U_{E}(i)^\Delta \ar[r]& \underset{F\subset E \subset \tilde{F}(\mu_p)}{\varprojlim} U'_{E}(i)^\Delta \ar[r]&  \bigoplus_{v\in S_p(F)} \Z_{p}[[\tilde{\Gamma}_F/\tilde{\Gamma}_{F,v}]](i)^\Delta \ar  `[d]   `[d] [d]\\ & 0 &  X'(\tilde{F}(\mu_p))(i)^\Delta \ar[l]& X(\tilde{F}(\mu_p))(i)^\Delta  \ar[l] }
	\end{equation}
	Let $\Delta_v$ be the decomoposition subgroup of a fixed place of $\tilde{F}(\mu_p)$ above $v$ in the extension $\tilde{F}(\mu_p)/\tilde{F}.$ Since $\Delta_v$ is a subgroup of $\tilde{\Gamma}_{F,v},$ it acts trivially on $ \Z_{p}[[\tilde{\Gamma}_F/\tilde{\Gamma}_{F,v}]].$ Then for any  $f\in \Z_{p}[[\tilde{\Gamma}_F/\tilde{\Gamma}_{F,v}]](i),$
	\begin{equation*}
		f\in \Z_{p}[[\tilde{\Gamma}_F/\tilde{\Gamma}_{F,v}]](i)^{\Delta_v} \Longleftrightarrow \omega^i(\sigma)f=f \mbox{ for all } \sigma \mbox{ in } \Delta_v,
	\end{equation*}
	 where $\omega$ is the Teichm\"{u}ller character. Thus $\Z_{p}[[\tilde{\Gamma}_F/\tilde{\Gamma}_{F,v}]](i)^{\Delta_v}$ is zero exactly when $\omega^i(\Delta_v) \neq 1.$ We know that 
	\begin{equation*}
	\Z_{p}[[\tilde{\Gamma}_F/\tilde{\Gamma}_{F,v}]](i)^{\Delta}   \subset \Z_{p}[[\tilde{\Gamma}_F/\tilde{\Gamma}_{F,v}]](i)^{\Delta_v}.
	\end{equation*}
	Hence if $\omega^i(\Delta_v) \neq 1,$ then $\Z_{p}[[\tilde{\Gamma}_F/\tilde{\Gamma}_{F,v}]](i)^{\Delta} $ is zero.\\
	Suppose now that $\omega^i(\Delta_v) = 1,$ then
	\begin{equation*}
	\Z_{p}[[\tilde{\Gamma}_F/\tilde{\Gamma}_{F,v}]](i)^{\Delta}   = \Z_{p}[[\tilde{\Gamma}_F/\tilde{\Gamma}_{F,v}]](i)^{\Delta/\Delta_v}.
	\end{equation*}
    Let $\tilde{\Gamma}_v$ be the decomposition group of $v$ in the extension $\tilde{F}/F.$ Notice that we have the following decomposition 
    \begin{equation*}
     \tilde{\Gamma}_F/\tilde{\Gamma}_{F,v} = \Delta/\Delta_v \times \tilde{\Gamma}/\tilde{\Gamma}_v.
     \end{equation*} 
     Then we obtain
	\begin{equation*}
	\Z_{p}[[\tilde{\Gamma}_F/\tilde{\Gamma}_{F,v}]](i)^{\Delta/\Delta_v} \simeq \Z_{p}[[\tilde{\Gamma}/\tilde{\Gamma}_v]].
	\end{equation*}
	Therefore, the pseudo-nullity of the $\tilde{\Lambda}$-module $\Z_{p}[[\tilde{\Gamma}_F/\tilde{\Gamma}_{F,v}]](i)^{\Delta}$ is implied by the condition (Dec), in this case.\\
	Then we conclude that under the condition (Dec), $\Z_{p}[[\tilde{\Gamma}_F/\tilde{\Gamma}_{F,v}]](i)^{\Delta}$ is pseudo-null over $\tilde{\Lambda}.$
	 Therefore, the exact sequence \eqref{eq seq etale} gives the desired pseudo-isomorphism. 
\end{proof}
In \cite{Equi GGC}, the authors show many formulations for Greenberg generalized conjecture (see also \cite[Section 2.2]{Hubbard}). More precisely, Proposition 3.6 in \cite{Equi GGC} establishes, for any number field $F$ satisfying the condition (Dec), the equivalence between the validity of $GGC'$ and the triviality of $\Tor{\tilde{\Lambda}}{\XS{-1}{\tilde{F}(\mu_p)}}^{\Gal{F(\mu_p)/F}}.$ Here we give a twisted version of this equivalence. We start with the following technical lemma.
\begin{lemm}\label{pseudo-null}
	Let $F$ be a number field. Then for any integer $i$ we have an isomorphism
	\begin{equation*}
	\varprojlim \bigoplus_{v\in S(E)} H^2(E_v, \Z_p(i)) \simeq \bigoplus_{\underset{i\equiv 1 \mod{[F_v(\mu_p):F_v]}}{v\in S(F)}} \ind{\Gamma}{\Gamma_v} H^0((\tilde{F})_v, \Q_p/\Z_p(1-i))^\vee,
	\end{equation*}
	where the inverse limit is taken via the corestriction maps, and $(\tilde{F})_v$ is the completion of $\tilde{F}$ at a fixed prime above $v\in S(F).$ In particular, if $F$ satisfies  the condition (Dec), then $ \varprojlim \displaystyle\bigoplus_{v\in S(E)} H^2(E_v, \Z_p(i))$ is pseudo-null.
\end{lemm}
\begin{proof}
	For any finite subextension $E/F$ of $\tilde{F}/F,$ let $G=\Gal{E/F}$ and $G_{v}$ be the decomposition subgroup of $v\in S$ in the extension $E/F.$ We have 
	\begin{align*}
	\varprojlim \bigoplus_{v\in S(E)} H^2(E_v, \Z_p(i)) 
	&= \varprojlim \bigoplus_{v\in S(F)} \bigoplus_{w\mid v} H^2(E_w, \Z_p(i)) \\
	&= \bigoplus_{v\in S(F)} \varprojlim \ind{G}{G_v} H^2(E_v, \Z_p(i)). \\ 
	\end{align*}
	By local duality we have 
	\begin{equation*}
           H^2(E_v, \Z_p(i)) \simeq H^0(E_v, \Q_p/\Z_p(1-i))^\vee.
	\end{equation*}
	Therefore 
	\begin{equation}\label{sum equality}
		\varprojlim \bigoplus_{v\in S(E)} H^2(E_v, \Z_p(i))  \simeq \bigoplus_{v\in S(F)} \varprojlim \ind{G}{G_v} (H^0(E_v, \Q_p/\Z_p(1-i))^\vee).
	\end{equation}
	Recall that 
	\begin{equation*}
	\ind{G}{G_v} (H^0(E_v, \Q_p/\Z_p(1-i))^\vee) = (\coind{G}{G_v} H^0(E_v, \Q_p/\Z_p(1-i)))^\vee. 
	\end{equation*}
	Hence the isomorphism $(\ref{sum equality})$ becomes
	\begin{equation*}
	\varprojlim \bigoplus_{v\in S(E)} H^2(E_v, \Z_p(i)) \simeq \bigoplus_{v\in S(F)} (\varinjlim \coind{G}{G_v} H^0(E_v, \Q_p/\Z_p(1-i)))^\vee 
	\end{equation*}
	Let's denote $\tilde{\Gamma}=\Gal{\tilde{F}/F},$ $\tilde{\Gamma}_E=\Gal{\tilde{F}/E},$ $\tilde{\Gamma}_v=\Gal{(\tilde{F})_v/F_v}$ and $\tilde{\Gamma}_{v,E}=\Gal{(\tilde{F})_v/E_v}.$ Observe that 
	\begin{equation*}
	G=\tilde{\Gamma}/\tilde{\Gamma}_E  \:\mbox{  and   }\: G_v=\tilde{\Gamma}_v/\tilde{\Gamma}_{v,E}=\tilde{\Gamma}_E\tilde{\Gamma}_v/\tilde{\Gamma}_E,
	\end{equation*}
	and
	\begin{equation*}
	H^0(E_v, \Q_p/\Z_p(1-i)) = H^0((\tilde{F})_v, \Q_p/\Z_p(1-i))^{\tilde{\Gamma}_{v,E}}.
	\end{equation*}
	Using  \cite[Proposition 6.10.4]{Zariski}, we conclude that 
	\begin{align*}
	\varprojlim \bigoplus_{v\in S_p(E)} H^2(E_v, \Z_p(i))
	&= (\bigoplus_{v\in S_p(F)} \coind{\tilde{\Gamma}}{\tilde{\Gamma}_v} H^0((\tilde{F})_v, \Q_p/\Z_p(1-i)))^\vee \nonumber \\
	&= \bigoplus_{v\in S(F)} \ind{\tilde{\Gamma}}{\tilde{\Gamma}_v} (H^0((\tilde{F})_v, \Q_p/\Z_p(1-i))^\vee).
	\end{align*}
	By the definition of the induced module we obtain that 
	\begin{equation*}
		\varprojlim \bigoplus_{v\in S(E)} H^2(E_v, \Z_p(i)) = \bigoplus_{v\in S(F)}
		 H^0((\tilde{F})_v, \Q_p/\Z_p(1-i))^\vee \hat{\otimes}_{\Z_p[[\tilde{\Gamma}_v]]} \tilde{\Lambda}.
	\end{equation*}
	Let $v\in S(F),$ the field $(\tilde{F})_v$ contains the cyclotomic $\Z_p$-extension of $F_v.$ For such a $v$ there are two possible cases
	\begin{enumerate}
		\item  $i \equiv 1 \mod{[F_v(\mu_p):F_v]}.$ In this case the absolute Galois group $G_{(\tilde{F})_v}$ acts trivially on $\Q_p/\Z_p(1-i).$ Hence 
		\begin{equation*}
		H^0((\tilde{F})_v, \Q_p/\Z_p(1-i))= \Q_p/\Z_p(1-i).
		\end{equation*}
		\item  $i \not\equiv 1 \mod{[F_v(\mu_p):F_v]}$). In this case the action of $G_{(\tilde{F})_v}$ on  $\Q_p/\Z_p(1-i)$ is not trivial and the cohomological group $H^0((\tilde{F})_v, \Q_p/\Z_p(1-i))$ is zero.
	\end{enumerate}
	Thus we conclude that
	\begin{equation*}
	\varprojlim \bigoplus_{v\in S(E)} H^2(E_v, \Z_p(i)) = \bigoplus_{\overset{v\in S(F)}{i \equiv 1 \mod{[F_v(\mu_p):F_v]}}}
	\Z_p(i-1) \hat{\otimes}_{\Z_p[[\tilde{\Gamma}_v]]} \tilde{\Lambda}.
	\end{equation*}
	Hence, using \cite[Proposition $5.8.1$]{Zariski}, we obtain
		\begin{equation*}
	\varprojlim \bigoplus_{v\in S(E)} H^2(E_v, \Z_p(i)) \simeq \bigoplus_{\overset{v\in S(F)}{i \equiv 1 \mod{[F_v(\mu_p):F_v]}}}
	 \Z_p[[\tilde{\Gamma}/\tilde{\Gamma}_v]](i-1).
	\end{equation*}
	Then if $F$ satisfies the condition (Dec), the $\tilde{\Lambda}$-module $\varprojlim \bigoplus_{v\in S(E)} H^2(E_v, \Z_p(i))$ is pseudo-null.
\end{proof}
The proof of Lemma \ref{pseudo-null} leads to the following remark.
\begin{rema}\label{without Dec}
Let $F$ be a number field, and let $i$ be an integer. If for each place $v\in S_p$ we have $i \not\equiv 1 \mod{[F_v(\mu_p):F_v]}$, then  
\begin{equation*}
\varprojlim \bigoplus_{v\in S(E)} H^2(E_v, \Z_p(i)) = 0.
\end{equation*}
\end{rema}
\noindent Now we can generalize Proposition 3.6 in \cite{Equi GGC} to any twist $i\not \equiv 0 \mod{[F(\mu_p): F]}.$
\begin{theo}\label{formulation Greenberg conjecture}
	Let $i$ be an integer such that $i\not\equiv 0 \mod{[F(\mu_p):F]}$ . Let $F$ be a number field for which the condition $(Dec)$ holds, the following assertions are equivalent:
	\begin{enumerate}
		\item $\Tor{\tilde{\Lambda}}{\XS{-i}{\tilde{F}}} = 0.$
		\item $X'(\tilde{F}(\mu_p))(i-1)^\Delta$ is pseudo-null.
		\item $X(\tilde{F}(\mu_p))(i-1)^\Delta$ is pseudo-null.
	\end{enumerate}
\end{theo}
\begin{proof}
	The equivalence between $(2)$ and $(3)$ is a consequence of Proposition \ref{proposition chi composante}, so we prove here only the equivalence between $(1)$ and $(2).$\\
    We start by remarking that if $i\equiv j \mod{[F(\mu_p):F]}$ we have 
    \begin{equation*}
    \Tor{\tilde{\Lambda}}{\XS{-i}{\tilde{F}}}=\Tor{\tilde{\Lambda}}{\XS{-j}{\tilde{F}}}(j-i), 
    \end{equation*} 
    and
    \begin{equation*}
    	X(\tilde{F}(\mu_p))(i-1)^\Delta=X(\tilde{F}(\mu_p))(j-1)^\Delta (i-j).
    \end{equation*}
    Then we can suppose that $i\geq 2$ to be sure that $LC_i$ is satisfied.\\
    Let $E/F$ be a finite subextension of $\tilde{F}/F,$ the Poitou-Tate exact sequence writes 
	\begin{equation*}
	\xymatrix@=1pc{
		0 \ar[r] & \cker{E} \ar[r] & H^2(G_S(E),\Z_p(i)) \ar[r]^{loc\qquad}&  \displaystyle\bigoplus_{v\in S(E)} H^2(E_v,\Z_p(i)) \quad\qquad 	\ar@{-<}  `[d]   `[d] [d] \\
		 & & 0 &  H^0(G_S(E), \Q_p/\Z_p(1-i))^\vee  \ar[l]
	}
	\end{equation*}
	where $\cker{E}$ is the kernel of the localization map $loc.$
	Taking the inverse limit over all finite subextensions of $\tilde{F}/F$ with respect to the corestriction maps, we obtain the following exact sequence 
	\begin{equation*}
	\xymatrix@=1pc{
		0 \ar[r] & \displaystyle\varprojlim_{E} \cker{E} \ar[r] & \qquad \displaystyle\varprojlim_E H^2(G_S(E),\Z_p(i))  \qquad 	\ar  `[d]   `[d] [d]\\
		  0 & \displaystyle\varprojlim_E H^0(G_S(E), \Q_p/\Z_p(1-i))^\vee \ar[l] &  \displaystyle\varprojlim_E \bigoplus_{v\in S(E)} H^2(E_v,\Z_p(i))  \ar[l]
	}
	\end{equation*}
	Lemma \ref{pseudo-null} shows that the $\tilde{\Lambda}$-module $\displaystyle\varprojlim_E \bigoplus_{v\in S(E)} H^2(E_v,\Z_p(i))$ is pseudo-null, hence we have a pseudo isomorphism
	\begin{equation}\label{pseudo isomorphism}
	\varprojlim_E \cker{E} \sim \varprojlim_E H^2(G_S(E),\Z_p(i)).
	\end{equation}
	Furthermore, it is well known (practically by definition) that 
	\begin{equation*}
	\varprojlim_E \cker{E} \simeq X'(\tilde{F}(\mu_p))(i-1)^\Delta.
	\end{equation*}
	Then the pseudo isomorphism (\ref{pseudo isomorphism}) becomes 
	\begin{equation}\label{pseudo isomorphism class group}
	X'(\tilde{F}(\mu_p))(i-1)^\Delta \sim \varprojlim_E H^2(G_S(E),\Z_p(i)).
	\end{equation}
	For any $\tilde{\Lambda}$-module $M$ and any positive integer $k,$ let (\cite[Definition 1.4]{Jannsen2})
	\begin{equation*}
		\mathrm{E}^{k}(M):=\mathrm{Ext}_{\tilde{\Lambda}}^{k}(M, \tilde{\Lambda}).
	\end{equation*}
	In particular, $\mathrm{E}^1(M)$ is the adjoint of $M$ in Iwasawa theory and $\mathrm{E}^0(M):= M^+=\mathrm{Hom}(M, \tilde{\Lambda})$ (\cite[page 16]{Perin Riou}, \cite[Lemma $3.1$]{Jannsen2}).\\
	The condition $i\not\equiv 0 \mod{[F(\mu_p):F]}$ implies that the action of $G_S(\tilde{F})$ on $\Q_p/\Z_p(i)$ is not trivial so that
	\begin{equation*}
	H^0(G_S(\tilde{F}), \Q_p/\Z_p(i)) = 0.
	\end{equation*}
	Then the exact sequence $b)$ of \cite[Corollary 3]{Jannsen} for the discrete $G_S(\tilde{F})$-module $\Q_p/\Z_{p}(i)$ writes
	\begin{equation}\label{Jannsen exact sequence}
	\xymatrix@=1pc{0 \ar[r]&  \mathrm{E}^1(H^1(G_S(\tilde{F}),\Q_p/\Z_{p}(i))^\vee) \ar[r]& \varprojlim_E H^2(G_S(E),\Z_p(i)) \quad \: \ar@{<}  `[d]   `[d] [d]\\
	0  & \mathrm{E}^2(H^1(G_S(\tilde{F}),\Q_p/\Z_{p}(i))^\vee) \ar[l] &(H^2(G_S(\tilde{F}), \Q_p/\Z_p(i))^\vee)^+ 	\ar[l]}
	\end{equation}
	By the twisted weak Leopoldt conjecture ($i\geq 2$) we have 
	\begin{equation*}
	H^2(G_S(\tilde{F}), \Q_p/\Z_p(i)) = 0.
	\end{equation*}
	Thus the exact sequence (\ref{Jannsen exact sequence}) gives the following isomorphism
	\begin{equation*}
	\mathrm{E}^1(\XS{-i}{\tilde{F}}) \simeq \varprojlim_E H^2(G_S(E),\Z_p(i)).
	\end{equation*} 
	This leads to an isomorphism 
	\begin{equation*}
	\mathrm{E}^1(\mathrm{E}^1(\XS{-i}{\tilde{F}})) \simeq \mathrm{E}^1(\varprojlim_E H^2(G_S(E),\Z_p(i))).
	\end{equation*}
	On the one hand, by (\ref{pseudo isomorphism class group}) $\varprojlim_E H^2(G_S(E),\Z_p(i))$ is a torsion $\tilde{\Lambda}$-module, so using $(iii)$ in \cite[Proposition $8,$ Chap. I]{Perin Riou} we have
	\begin{equation}\label{pseudo isomorphism limite proj}
	\mathrm{E}^1(\varprojlim_E H^2(G_S(E),\Z_p(i))) \sim \varprojlim_E H^2(G_S(E),\Z_p(i)).
	\end{equation}
	On the other hand, Theorem \ref{multiple projective dimension} shows that $\XS{-i}{\tilde{F}}$ is a $\tilde{\Lambda}$-module of projective dimension at most $1.$ Then Theorem $1.6$ and Lemma $1.8$ of \cite{Jannsen2} show that $\mathrm{E}^1(\mathrm{E}^1(\XS{-i}{\tilde{F}}))$ is canonically isomorphic to the kernel of the natural map
	\begin{equation*}
		\xymatrix@=2pc{ \XS{-i}{\tilde{F}} \ar[r]& \XS{-i}{\tilde{F}}^{++},}
	\end{equation*}
	which is obviously $\Tor{\tilde{\Lambda}}{\XS{-i}{\tilde{F}}}.$ Hence we have
	\begin{equation*}
	\mathrm{E}^1(\mathrm{E}^1(\XS{-i}{\tilde{F}})) \simeq \Tor{\tilde{\Lambda}}{\XS{-i}{\tilde{F}}}.
	\end{equation*}
	Therefore, the pseudo isomorphims (\ref{pseudo isomorphism class group}) and (\ref{pseudo isomorphism limite proj}) give a pseudo isomorphism
	\begin{equation}\label{last ps-iso}
	\Tor{\tilde{\Lambda}}{\XS{-i}{\tilde{F}}} \sim X'(\tilde{F}(\mu_p))(i-1)^\Delta.
	\end{equation}
	Since $\XS{-i}{\tilde{F}}$ has no non trivial pseudo-null sub-module, the pseudo-nullity of $\Tor{\tilde{\Lambda}}{\XS{-i}{\tilde{F}}}$ is equivalent to its triviality.	Hence we get the equivalence between $(1)$ and $(2)$ by the pseudo-isomorphism \eqref{last ps-iso}.
\end{proof}
\begin{rema}
	\begin{enumerate}
		\item The hypothesis $i\not \equiv 0 \mod{[F(\mu_p):F]}$ ensures that the cohomological group $H^0(G_S(\tilde{F}), \Q_p/\Z_p(i))$ is zero, which is needed to apply Corollary $3$ in \cite{Jannsen}.
		\item If $F$ contains $\mu_p,$ the equivalence between $(1)$ and $(2)$ is given in \cite[ Proposition $3.6$]{Equi GGC} and for the other one (i.e $(2) \Leftrightarrow (3)$) see \cite[page $3$]{Nguyen19}.
	\end{enumerate}
\end{rema}
Recall that Theorem \ref{theorem torsion free} tells us that if $F$ have a $\Z_{p}^d$-extension $L$ for which $\Tor{\Lambda_L}{\XS{-i}{L}}=0$ and contains a $\Z_{p}$-extension satisfying $WLC_i,$ then $\Tor{\tilde{\Lambda}}{\XS{-i}{\tilde{F}}}=0.$ As a consequence of Theorem \ref{formulation Greenberg conjecture} the $\tilde{\Lambda}$-module $X(\tilde{F}(\mu_p))(i-1)^\Delta$ is pseudo-null if $F$ verifies the condition $(Dec)$ and $i\not\equiv 0 \mod{[F(\mu_p):F]}.$ These prove the following result.
\begin{theo}\label{main theorem}
	Assume that the condition $(Dec)$ holds for $F.$ Let $i$ be an integer such that $i\not\equiv 0 \mod{[F(\mu_p):F]}.$ If $F$ admits  a $\Z_{p}^d$-extension $L$ such that $\Tor{\Lambda_L}{\XS{-i}{L}}=0$ and $L$ contains a $\Z_{p}$-extension which satisfies $WLC_i,$ then the $\tilde{\Lambda}$-module $X(\tilde{F}(\mu_p))(i-1)^\Delta$ is pseudo-null. \hfill $\square$
\end{theo} 
In particular, if $i\equiv 1 \mod{[F(\mu_p):F]}$ we obtain a sufficient condition for the validity of Greenberg's generalized conjecture.
\begin{coro}\label{twist i=1}
	Assume that the condition $(Dec)$ holds for $F.$ Assume that there exist a $\Z_{p}^d$-extension $L$ and an integer $i\equiv 1 \mod{[F(\mu_p):F]}$ such that 
	\begin{enumerate}
		\item $L$ contains a $\Z_{p}$-extension which satisfies $WLC_i,$ and 
		\item  $\Tor{\Lambda_L}{\XS{-i}{L}}=0.$
	\end{enumerate}
	Then $F$ satisfies Greenberg generalized conjecture. \hfill $\square$
\end{coro}
\begin{rema}\label{Nguyen's result} 
	In \cite{Nguyen19}, Nguyen   Quang Do gave an inductive process for the pseudo-nullity of the unramified totally split Iwasawa module in a tower of multiple $\Z_{p}$-extensions. Precisely, he proved that GGC holds for any number field $F$ (see \cite[Theorem 1.4]{Nguyen19}) satisfying the following conditions:
	\begin{itemize}
		\item $F$ verifies Kuz'min-Gross conjecture (i.e $X'(F^c)^{\Gamma_c}$ is finite).
		\item $F$ admits a $\Z_{p}^2$-extension $F^{(2)}$ which is normal over $\Q,$ contains $F^c$ and $X'(F^{(2)})$ is pseudo-null.
	\end{itemize}
	Under these conditions, Lemma $1.2$ of \cite{Nguyen19} proves that the decomposition group of any $p$-place of $F^{(2)}$ has $\Z_p$-rank $2.$ \\
	By choosing $i\geq 2$ and $i\equiv 1\mod{[F(\mu_p):F]},$ Theorem \ref{formulation Greenberg conjecture} gives an equivalence between the pseudo-nullity of $X'(F^{(2)})$ and the triviality of $\Tor{\Lambda_{(2)}}{\XS{-i}{F^{(2)}}}.$ Then, the $\Z_{p}^2$-extension $F^{(2)}$ satisfies the conditions of Theorem \ref{theorem torsion free}. Hence, we obtain Nguyen Quang Do's result.
\end{rema}

\section{$(p,i)$-regular number fields}

The $(p, i)$-regular number fields were introduced in \cite{Assim95}, they are a twisted generalization of $p$-rational fields (see \cite{Movahhedi90} and \cite{Movahhedi-Nguyen}). A number field $F$ is called $(p, i)$-regular if the second cohomological group $H^2(G_S(F), \Z_p/p(i))$ is trivial. In this section, we show that for any $(p, i)$-regular number field $F,$ the $\tilde{\Lambda}$-torsion sub-module of $\XS{-i}{\tilde{F}}$ vanishes. As a consequence, we obtain that for $(p, i)$-regular number fields which satisfy the condition $(Dec),$ the $\tilde{\Lambda}$-module $X(\tilde{F}(\mu_p))(i-1)^\Delta$ is pseudo-null.\\
 There are many equivalent properties to the $(p, i)$-regularity. We cite here the following equivalence
\begin{equation*}
	F \mbox{ is } (p,i)\mbox{-regular } \Longleftrightarrow  \mbox{ The module } \XS{-i}{F^c} \mbox{ is } \Lambda_c\mbox{-free}.
\end{equation*}
This equivalence was proved in \cite{Nguyen19} for integers $i\neq 1.$ The case $i=1$ is a consequence of the following periodicity, for $i\equiv j \mod{[F(\mu_p):F]},$
\begin{center}
	$F$ is $(p,i)$-regular $\Longleftrightarrow$ $F$ is $(p,j)$-regular.
\end{center}
	Since, for any integer $j\neq 1$ and $j\equiv 1 \mod{[F(\mu_p):F]},$ we have
	\begin{equation*}
		\XS{-1}{F^c} = \XS{-j}{F^c}(j-1).
	\end{equation*}
In particular, any $(p, i)$-regular number field $F$  satisfies $\Tor{\Lambda_{c}}{\XS{-i}{F^c}}=0.$ Then as a consequence of Theorem \ref{theorem torsion free} we get 
\begin{prop}\label{corollary torsion free}
	Let $i$ be an integer. For any $(p, i)$-regular number field $F$ we have 
	\begin{equation*}
	\Tor{\tilde{\Lambda}}{\XS{-i}{\tilde{F}}}=0.
	\end{equation*}
	\hfill $\square$
\end{prop}
\begin{rema}
     In \cite[Theorem $1$]{Fujii}, Fujii shows that $\Tor{\tilde{\Lambda}}{\XS{0}{\tilde{F}}}=0$ for a family of quadratic fields $F$ which are $p$-rational. The above proposition affirms that this result is true for all $p$-rational number fields and also gives a twisted generalization  of \cite[Theorem $1$]{Fujii}.
\end{rema}

Proposition \ref{corollary torsion free} and  Theorem \ref{formulation Greenberg conjecture} lead to the following consequence.
\begin{theo} \label{a part of GGC}
	Let $i \not \equiv 0 \mod{[F(\mu_p): F]}$ be an integer. For any imaginary $(p, i)$-regular number field which satisfies the condition (Dec), the $\tilde{\Lambda}$-module $X(\tilde{F}(\mu_p))(i-1)^\Delta$ is pseudo-null.
	\hfill $\square$
\end{theo}
In particular, when the integer $i\equiv 1 \mod{[F(\mu_p):F]}$ we get the following
\begin{coro}\label{regualr number field}
	Greenberg's generalized conjecture holds for all imaginary $(p, 1)$-regular number field which satisfies the condition (Dec).
	\hfill $\square$
\end{coro}
\noindent In \cite[Theorem $3.4$]{Nguyen19}, Nguyen Quang Do obtained a similar result for any integer $i.$ Precisely, he showed that $X'(\tilde{F}(\mu_p))(i-1)^\Delta$ vanishes for any $(p, i)$-regular number field $F.$ Assuming that $F$ satisfies the conditions (Dec), Proposition \ref{proposition chi composante} implies the pseudo-nullity of $X(\tilde{F}(\mu_p))(i-1)^\Delta.$

The condition $i \not \equiv 0 \mod{[F(\mu_p): F]}$ means that Theorem \ref{a part of GGC} does not concern $p$-rational number fields. Thus we can ask the following question:
\begin{center}
	\textit{Does $p$-rational number fields satisfy Greenberg's generalized conjecture?}
\end{center} 
It is known that GGC holds for $p$-rational field $F$ containing $\mu_p.$ It follows from Corollary \ref{regualr number field}, since $F$ satisfies (Dec) and is $(p, 1)$-regular. In general, the implication
\begin{center}
	$F$ is $p$-rational $\Rightarrow$ $F$ is $(p, 1)$-regular,
\end{center}
fails to be true. Thus, even if the condition (Dec) is satisfied, Corollary \ref{regualr number field} does not apply. The quadratic number field $F=\Q(\sqrt{-61})$ has one $3$-adic place ($3$ is inert) but its $S_3$-class number is $6$. Thus $F$ is not $(3, 1)$-regular. In contrast, since $\Q(\sqrt{183})$ has class number $2$, by \cite[Corollary $4.1.2$]{Greenberg16} the field $F$ is $3$-rational.\\

\textbf{Acknowledgment.} The authors would like to thank Thong Nguyen Quang Do for many helpful comments on an earlier version of this paper.

\bibliographystyle{alpha} %% 21

\end{document}